\documentclass[12pt]{amsart}

\usepackage{amsmath, amssymb}

\newtheorem{theorem}{Theorem}

\newtheorem{proposition}{Proposition}
\newtheorem{lemma}{Lemma}

\theoremstyle{definition}
\newtheorem{definition}{Definition}

\theoremstyle{remark}
\newtheorem{remark}{Remark}

\newcommand{\R}{{\mathbb R}}
\newcommand{\C}{{\mathbb C}}

\newcommand{\N}{{\mathbb N}}
\newcommand{\Z}{{\mathbb Z}}

\renewcommand{\a}{\alpha}
\renewcommand{\b}{\beta}

\makeindex             % used for the subject index
\begin{document}

\title[The Bergman kernel
on model domains]{The asymptotic behavior of 
the Bergman kernel \\
on pseudoconvex model domains }
\author{Joe Kamimoto}
\address{Kyushu university, 
Motooka 744, Nishi-ku, 
Fukuoka, 819-0395, Japan.}  
\email{joe@math.kyushu-u.ac.jp}
%\and Name of Second Author \at Name, Address of Institute \email{name@email.address}}
\maketitle

\begin{abstract}
In this paper, we investigate the asymptotic 
behavior of the Bergman kernel at the boundary
for some pseudoconvex model domains. 
This behavior can be described by 
the geometrical information of the Newton polyhedron
of the defining function of the respective domains. 
We deal with not only the finite type cases 
but also some infinite type cases. 
\end{abstract}

\section{Introduction}

Let $\Omega$ be a domain in $\C^n$ and 
let $A^2(\Omega)$ be the Hilbert space 
of the $L^2$-holomorphic functions on $\Omega$. 
The {\it Bergman kernel} $B_{\Omega}(z)$ of $\Omega$ 
(on the diagonal) is defined by 
$B_{\Omega}(z)=\sum_{\alpha} |\phi_{\alpha}(z)|^2$, 
where $\{\phi_{\alpha}\}$ is a complete orthonormal 
basis of $A^2(\Omega)$. 
Throughout this paper, 
we assume that the boundary $\partial\Omega$ 
of $\Omega$ is always $C^{\infty}$-smooth.

Since the behavior of the Bergman kernel 
at the boundary plays essentially important roles in
the study of several complex variables and 
complex geometry, 
many interesting results about its behavior  
have been obtained. 

In the case of strictly pseudoconvex domains, 
a beautiful asymptotic expansion of the Bergman 
kernel was given by C. Fefferman \cite{Fef74}
and Boutet de Monvel and Sj\"{o}strand \cite{BoS76}.
In the case of weakly pseudoconvex domains,  
many kinds of important results have been obtained
(see the references in \cite{Kam04}, \cite{ChF12}, etc.). 
In particular, 
in the two-dimensional finite type case,  
an asymptotic expansion analogous to 
that of Fefferman was recently given 
by Hsiao and Savale \cite{HsS22}.
%(\cite{Cat89}, \cite{BSY95}, \cite{HsS22}). 
%
On the other hand, 
in the higher dimensional case, 
there does not seem to be 
such strong and general results. 
%seems to be not so strong and general results. 
In \cite{Kam04}, the author investigated 
a special case of pseudoconvex model domains 
and computed some asymptotic 
expansion of the Bergman kernel.  
%(see also \cite{CKO04}). 
The purpose of this paper is to 
generalize the results in \cite{Kam04}.

In \cite{Kam04}, only the finite type case was 
dealt with, while more general cases 
will be considered in this paper; for example,
some infinite type cases can be also dealt with.
Some two-dimensional infinite type cases have been 
precisely investigated in \cite{Bha10}, \cite{Bha20}.
We will consider higher dimensional cases, 
which are more complicated.  

From the results in \cite{Kam04}, \cite{CKO04},
\cite{ChL09}, \cite{CKN11}, \cite{ChF12}, 
it might be recognized that 
the information of the boundary 
from the viewpoint of singularity theory 
is valuable for the exact analysis of the Bergman kernel 
in the higher dimensional weakly 
pseudoconvex case. 
In particular, the {\it Newton polyhedron} 
determined from the boundary contains 
fruitful information for 
the singularity of the Bergman kernel 
at the boundary. 

One of the difficulty of the analysis 
in the infinite type case
is caused by the existence
of non-zero {\it flat functions}  
(see Section 2.1). 
Notice that flat functions do not 
affect the geometry of the Newton polyhedron
and the influence of flat functions
is subtle in the singularity of 
the Bergman kernel. 
However, 
this practical influence cannot always be negligible. 
In the main theorem (Theorem~2),  
we will give a certain condition 
on the geometry of the Newton polyhedron
to determine the case 
where the above influence of 
flat functions can be negligible
in some sense.

This paper is organized as follows. 
In Section~2, 
we state a main theorem 
and explain its significance. 
In Section~3, 
we exhibit
an integral formula 
of the Bergman kernel
given by F. Haslinger \cite{Has94}, \cite{Has98},
on which our analysis is based. 
In Section~4, 
we show that the singularity 
of the Bergman kernel can be 
completely determined by 
the local geometry of the boundary. 
In our analysis, it is necessary
to consider various kinds of $C^{\infty}$ functions, but 
the $C^{\infty}$ class contains 
many troublesome functions.
In \cite{KaN16jmst}, \cite{KaN16tams}, 
a certain class of $C^{\infty}$ functions, 
called the $\hat{\mathcal E}$-class, is 
introduced by the use of Newton polyhedra, 
which is easy to deal with.
Moreover, we also introduce a class analogous to 
the $\hat{\mathcal E}$-class
in the complex variables case 
in Section~5. 
In Section~6, 
we investigate 
the singularity of the Bergman kernel 
in the $\hat{\mathcal E}$-case. 
In Section~7, 
a main theorem can be shown 
by the use of the results 
in the $\hat{\mathcal E}$-case.
By the way, 
the behavior of 
some Laplace type integrals 
is a key in the analysis in Section~6. 
The work of Varchenko in \cite{Var76}
(see also \cite{AGV88}) 
concerning local zeta functions and 
oscillatory integrals plays crucial roles
in the investigation of the above behavior.
In Section~8, 
we explain some results in \cite{CKN13},  
\cite{KaN16jmst}, \cite{KaN16tams}, 
which generalize the above 
Varchenko's results, and 
apply these results to the anlaysis 
of the Bergman kernel. 
In the last section, 
we will explain some important words and concepts.

\bigskip

{\it Notation and symbols.}\quad

%%%%%%%%%%%%%%%%%%
\begin{itemize}%{enumerate}
\item
We denote by $\N$, $\Z$, $\R$, $\C$ the set consisting of 
all natural numbers, integers, real numbers, 
complex numbers, respectively. 
Moreover, we denote by $\Z_{+}$, $\R_{+}$  
the set consisting of 
all nonnegative integers, real numbers, respectively.
For $s\in\C$, ${\rm Im}(s)$ expresses the imaginary part of $s$.
%%%
\item
Let 
$\alpha:=(\alpha_1,\ldots,\alpha_n)$,
$\beta:=(\beta_1,\ldots,\beta_n)\in\Z_+^n$. 
The multi-indeces will be used as follows.
For $x=(x_1,\ldots,x_n)\in\R^n$, 
define
%%%%%%
\begin{eqnarray*}
x^{\a}=x_1^{\a_1}\cdots x_n^{\a_n}, 
\quad 
\|x\|_{\R}=\sqrt{x_1^2+\cdots+x_n^2}.
\end{eqnarray*}
%%%%%%
For $z=(z_1,\ldots,z_n), \,\,
\bar{z}=(\bar{z}_1,\ldots,\bar{z}_n)
\in\C^n$, 
%%%
define
%%%%%%
\begin{eqnarray*}
%&& 
%\langle x,y\rangle =x_1 y_1+\dots+x_n y_n, 
%\\
&& 
z^{\alpha}:=z_1^{\a_1}\cdots z_n^{\a_n}, \,\,
\bar{z}^{\b}:=\bar{z}_1^{\b_1}\ldots\bar{z}_n^{\b_n}, \,\,
|z|^{2\alpha}:=|z_1|^{2\a_1}\cdots|z_n|^{2\a_n}, \\
&&\|z\|=\sqrt{|z_1|^{2}+\cdots +|z_n|^{2}}.
\end{eqnarray*}
%%%%%%%%%%%%%%%%%%%%%%%%
\item 
For a positive number $R$, we denote 
\begin{eqnarray*}
D_{\R}(R)=\{x\in\R^n: 
\|x\|_{\R}<R\},
\quad 
D(R)=\{z\in\C^n: \|z\|<R\}.
\end{eqnarray*}
\end{itemize}

%%%%%%%%%%%%%%%%%%%%%%%%%%%%%%%%%%%

\section{Main results}

\subsection{Newton data}
%%%%%
In this paper, 
many concepts in {\it convex geometry} play 
useful roles. 
We will explain the exact meanings of 
necessary words 
in convex geometry in Section~9.2 
(see also \cite{Zie95}).

\medskip

Let $F$ be a real-valued 
$C^{\infty}$ function defined near the origin in $\C^n$.
Let 
\begin{equation*}
\sum_{\alpha,\beta\in\Z_+^n}
C_{\alpha \beta} z^{\alpha} \bar{z}^{\beta}
=\sum_{\alpha,\beta\in\Z_+^n}
C_{\alpha \beta} z_1^{\alpha_1}\cdots z_n^{\alpha_n} 
\bar{z}_1^{\beta_1}\cdots \bar{z}_n^{\beta_n}
\end{equation*}
be the Taylor series of $F$ at the origin. 
The {\it support of} $F$ is the set
$S_F=\{\alpha+\beta\in\Z_+^n:
C_{\alpha\beta}\neq 0\}$ and the 
{\it Newton polyhedron of } $F$ is 
the integral polyhedron 
\begin{equation*}
{\mathcal N}_+(F)
=\mbox{ 
the convex hull of the set 
$\displaystyle \bigcup\{\alpha+\beta+\R_+^n:
\alpha+\beta \in S_F \}$ in $\R_+^n$.}
\end{equation*} 
We say 
that $F$ is {\it flat} if   
${\mathcal N}_+(F)=\emptyset$ and
that $F$ is {\it convenient}
if ${\mathcal N}_+(F)$ intersects 
all the axes.
For a compact face $\gamma$ of ${\mathcal N}_+(F)$, 
the $\gamma$-part of $F$ is defined by 
\begin{equation*}
F_{\gamma}(z)=
\sum_{\alpha+\beta\in\gamma\cap\Z_+^n}
C_{\alpha \beta} z^{\alpha} \bar{z}^{\beta}.
\end{equation*}

We define a quantity $\rho_F$ ($\in\Z_+$) 
as follows.
When $F$ is convenient, let
\begin{equation*}
\rho_F:=\max\{\rho_j(F):j=1,\ldots,n\},
\end{equation*}
where 
$$\rho_j(F):=\min\{t\geq0:
(0,\ldots,\stackrel{(j)}{t},\ldots,0)
%(0,\ldots,0, t, 0, \ldots,0)
\in {\mathcal N}_+(F)\}.
$$
When $F$ is not convenient, let 
$\rho_F:=\infty$.

Hereafter, we assume that 
$F$ is not flat. 
The {\it Newton distance of} $F$ is 
the nonnegative number 
\begin{equation*}
d_F:=\min\{t\geq 0 :(t,\ldots,t)\in 
{\mathcal N}_+(F)\}.
\end{equation*}
Since $F$ is not flat, 
there exists the minimum proper face of 
the Newton polyhedron ${\mathcal N}_+(F)$ containing 
the point $P_F:=(d_F,\ldots,d_F)$, which is called 
the {\it principal face} of ${\mathcal N}_+(F)$ and 
is denoted by $\gamma_*$. 
The codimension of $\gamma_*$ 
is called the {\it Newton multiplicity} 
of $F$, which is denoted by $m_F$
(i.e.,  $m_F=n-\dim(\gamma_*)$).
In particular, 
when $P_F$ is a vertex of ${\mathcal N}_+(F)$,
$\gamma_*$ is the point $P_F$ 
and $m_F=n$.
When $\gamma_*$ is compact, 
the {\it principal part} of $F$ is defined by 
$$
F_*(z)=%F_{\gamma_*}(z)=
\sum_{\alpha,\beta\in\gamma_*\cap\Z_+^n}
C_{\alpha \beta} z^{\alpha} \bar{z}^{\beta}
$$
(i.e., the principal part of $F$ is the $\gamma_*$-part
of $F$).

%%%%%%%%%%%%%%%%%%
\subsection{Main results}

Let $U$ be a complete pseudoconvex Reinhardt domain
in $\C^n$ (possibly $U=\C^n$).
%(i.e., if $z\in U$, then
%$(e^{i\theta_1} z_1,\ldots, e^{i\theta_n} z_n)
%\in U$ for any $\theta_j\in\R$).
Let $F$ be a real-valued $C^{\infty}$ function 
on $U$ satisfying the following conditions 
\begin{enumerate}
\item[(A)] {$F(z)=0$ if and only if $z=0$} 
and $F$ is not flat at the origin;
\item[(B)] $F$ is a plurisubharmonic function on $U$;
\item[(C)] 
$F(e^{i\theta_1} z_1,\ldots, e^{i\theta_n} z_n)
=F(z_1,\ldots,z_n)$ for any $\theta_j\in\R$ 
and $z\in U$;
\item[(D)]
If $U$ is unbounded, then 
there are some positive numbers $c$, $\beta$, $L$ 
such that $F(z)\geq c \|z\|^{\beta}$ 
for $z\in U\setminus
D(L)$.
\end{enumerate}

We will mainly deal with pseudoconvex model domains 
in $\C^{n+1}$ of the form 
$$
\Omega_F
=\{
(z_0, z_1, \ldots, z_n)\in \C \times U:
{\rm Im}(z_0)>F(z_1,\ldots, z_n)
\}.
$$

Note that the condition (D) implies that 
the dimension of $A^2(\Omega_F)$ is infinity.

In the case of the domain $\Omega_F$, 
the finite type condition can be easily 
seen in the information of the Newton polyhedron
of $F$. 
Let $\Delta_1(\partial\Omega_F,0)$ be 
the D'Angelo type
of $\partial\Omega_F$ at the origin.

\begin{lemma}[\cite{Kam21jmsj}]
$\Delta_1(\partial\Omega_F, 0)=\rho_F$.
In particular, 
the following two conditions are equivalent.
\begin{enumerate}
\item{$\partial\Omega_F$ is of finite type at $0$
(i.e., $\Delta_1(\partial\Omega_F,0)<\infty$);}
\item{$F$ is convenient 
(i.e., ${\mathcal N}_+(F)$ intersects all axes).}
\end{enumerate}
\end{lemma}
%%%%%%%%%%%%%%%%%%%%%%%%%%%%%%%%

Let $B_{\Omega_F}(z_0,z)=B_{\Omega_F}(z_0,z_1,\ldots,z_n)$ 
be the Bergman kernel 
(on the diagonal) of $\Omega_F$. 
Since we are interested in 
the behavior of 
the restriction of the Bergman kernel
of $\Omega_F$ to the vertical line, 
we define
\begin{equation*}
{\mathcal B}_F(\rho):=B_{\Omega_F}(0+i\rho,0,\ldots,0)
\quad \mbox{ for $\rho>0$.} 
\end{equation*} 
The behavior of ${\mathcal B}_F(\rho)$, 
as $\rho$ tends to $0$,  can be exactly 
expressed by the use of the geometry of 
the Newton polyhedron ${\mathcal N}_+(F)$ 
of $F$.

\begin{theorem}[\cite{Kam04}]
If $\partial\Omega_F$ is of finite type at $0$, then 
\begin{equation}\label{eqn:1}
{\mathcal B}_F(\rho)=
\frac{\Psi(\rho)}{\rho^{2/d_F+2}(\log \rho)^{m_F-1}}, 
\end{equation}
where $\Psi(\rho)$ admits the asymptotic expansion:
\begin{equation}\label{eqn:2}
\Psi(\rho) \sim \sum_{j=0}^{\infty}
\sum_{k=a_j}^{\infty} 
C_{jk}\frac{\rho^{j/m}}{(\log(1/\rho))^{k}}
\quad \mbox{as $\rho\to 0$,} 
\end{equation}
where $m$ is a positive integer, 
$a_j$ are integers and 
$C_{jk}$ are real numbers
(the exact meaning of the above asymptotic 
expansion is explained in Section~9.1, below).

Furthermore, the first coefficient of the above 
expansion can be 
determined by the use of the Newton data of $F$ 
as follows:
\begin{equation}\label{eqn:3}
{\mathcal B}_F(\rho)=
\frac{C(F_*)}{\rho^{2/d_F+2}(\log \rho)^{m_F-1}}
\cdot
(1+o(\rho^{\varepsilon})) 
\quad \mbox{as $\rho\to 0$,} 
\end{equation}
where $\varepsilon$ is a positive constant and 
$C(F_*)$ is a positive constant depending only 
on the principal part $F_*$ of $F$. 
\end{theorem}
%%%%%%%%%%%%%%%%%%%%%%%%%%%%%%

In this paper, 
we improve the above theorem in some sense as follows. 

%%%%%%%%%%%%%%%%%%%%%%%%%%%%%%%
\begin{theorem}
Suppose that there exists a $C^{\infty}$ function
$F_0$ defined near the origin in $\C^n$
such that $F_0$ satisfies the condition (C) and the 
$\hat{\mathcal E}$-condition at the origin 
(see Section~5)  and $F(z)-F_0(z)$ is a nonnegative flat function.  
If the principal face $\gamma_*$ of the Newton polyhedron 
of $F$ is compact, then 
\begin{equation}\label{eqn:4}
{\mathcal B}_F(\rho)=
\frac{C(F_*)}{\rho^{2/d_F+2}(\log \rho)^{m_F-1}}
\cdot
(1+o(\rho^{\varepsilon})) 
\quad \mbox{as $\rho\to 0$,} 
\end{equation}
where $\varepsilon$ is a positive constant and 
$C(F_*)$ is a positive constant depending only 
on the principal part $F_*$ of $F$. 
\end{theorem}

\begin{remark}
(1)\quad 
Note that  
the finite type condition is equivalent to 
the convenience condition of $F$ from Lemma~1. 
When $F$ is convenient, 
$F$ itself satisfies the $\hat{\mathcal E}$-condition
(see \cite{KaN16jmst}, \cite{KaN16tams}) 
and the principal face of ${\mathcal N}_+(F)$ 
is always compact.
Therefore, the assumption of Theorem 2
is weaker than that of Theorem 1. 
Indeed, 
Theorem~2 can be applied to some infinite type cases.

\smallskip

(2)\quad 
In the two-dimensional case 
(i.e. $\Omega_F\subset\C^2$), 
Lemma 1 implies that the finite type condition 
is equivalent to the nonflatness of $F$.
Since we assume 
the nonflatness of $F$ in (A),  
the advantage of Theorem~2 can be only seen in 
the case where 
the dimension is higher than two.

\smallskip

(3)\quad
Theorem~2 can be applied to the following examples, 
which are in the infinite type case.
\begin{itemize}
\item
When 
$F(z_1,z_2)=|z_1|^6+ |z_1|^2|z_2|^4+ e^{-1/|z_2|^2}$
near the origin, 
$$
{\mathcal B}_F(\rho)=
\frac{C(F_*)}{\rho^{8/3}}\cdot(1+o(\rho^{\varepsilon})) 
\quad \mbox{as $\rho\to 0$.} 
$$
(Note that $F_*(z_1,z_2)=|z_1|^6+|z_1|^2 |z_2|^4$.)
\smallskip
%%%%%%%%%%%%%%%%%
\item
When
$F(z_1,z_2)=|z_1|^6+ |z_1|^2|z_2|^2+ e^{-1/|z_2|^2}$
near the origin, 
$$
{\mathcal B}_F(\rho)=
\frac{C(F_*)}{\rho^{3} \log \rho}\cdot(1+o(\rho^{\varepsilon})) 
\quad \mbox{as $\rho\to 0$.} 
$$
(Note that $F_*(z_1,z_2)=|z_1|^2 |z_2|^2$.)
\smallskip
%%%%%%%%%%%%%%%%%%%
\item
When
$F(z_1,z_2)=|z_1|^2|z_2|^2+ e^{-1/|z_1|^2}+ e^{-1/|z_2|^2}$
near the origin, 
$$
{\mathcal B}_F(\rho)=
\frac{C(F_*)}{\rho^{3} \log \rho}\cdot(1+o(\rho^{\varepsilon})) 
\quad \mbox{as $\rho\to 0$.} 
$$
(Note that $F_*(z_1,z_2)=|z_1|^2 |z_2|^2$.)
\smallskip
%%%%%%%%%%%%%%%%%%%
\item
When
$F(z_1,z_2,z_3)=|z_1|^8+|z_2|^8+
|z_1|^2 |z_2|^2 |z_3|^2+ e^{-1/|z_3|^2}$
near the origin, 
$$
{\mathcal B}_F(\rho)=
\frac{C(F_*)}{\rho^{3} (\log \rho)^2} \cdot(1+o(\rho^{\varepsilon})) 
\quad \mbox{as $\rho\to 0$.} 
$$
(Note that $F_*(z_1,z_2)=|z_1|^2 |z_2|^2 |z_3|^2$.)
\end{itemize}
%%%%%%%%%%%%%%%%%%%%%%%%%%%%%%%%%%%%%%%
We remark that $e^{-1/|z_j|^2}$
($j=1, 2, 3$) in the above examples
can be replaced by any flat functions which
are positive away from the origin.

\smallskip

%%%%%%%%%%%%%%%%%%%%%%%%%%%%%%%%%%%%%%%%%

(4)\quad
Our method in the proof of Theorem 2 cannot 
generally show 
the existence of an asymptotic expansion 
of the form (\ref{eqn:2}). 
We guess that the pattern of the asymptotic 
expansion might not 
be expressed as in the form (\ref{eqn:2}) in general  
from some observation of the strange phenomena 
in \cite{KaN19}, \cite{KaN20}, \cite{Nos22}, 
which are seen in the analytic continuation 
of local zeta functions.

\smallskip

(5)\quad
In the case where the principal face is noncompact, 
there exist many examples in which the behavior 
of the Bergman kernel ${\mathcal B}_F(\rho)$ 
as $\rho\to 0$ is different from (\ref{eqn:4}).
For example, 
in the case where 
$F(z_1,z_2)=|z_1|^2+e^{-1/|z_2|^{p}}$ near the origin,
${\mathcal B}_F(\rho)$ locally satisfies 
$$
\frac{c_1 |\log \rho|^{1/p}}{\rho^3}
\leq 
{\mathcal B}_F(\rho)
\leq 
\frac{c_2 |\log \rho|^{1/p}}{\rho^3}
$$
near $\rho=0$, where 
$c_1,c_2$ are positive constants
(\cite{Kam23}). 
Notice that the logarithmic functions 
appear in the numerators
in the above estimates. 
Observing the above estimates, 
we can see that 
the behavior of ${\mathcal B}_F(\rho)$ 
depends on $p$, which is an information of the 
flat term $e^{-1/|z_2|^{p}}$. 
In the noncompact principal face case, 
the information of the Newton polyhedron 
of $F$ cannot always determine 
the singularity of the Bergman kernel completely.

\smallskip

(6)\quad 
Of course, the constant $\varepsilon$ 
in (\ref{eqn:3}), (\ref{eqn:4})
does not depend on $\rho$. 
The equations of similar type to (\ref{eqn:3}), (\ref{eqn:4})
will be often seen in this paper. 
In these equations, 
the constant $\varepsilon$ 
can be chosen by the use of the geometry of
the respective Newton polyhedron.

Since $\rho^{\varepsilon}=o((\log(1/\rho)^{-1})$ holds
for any ${\varepsilon}>0$,
$o(\rho^{\varepsilon})$ in the equations in (3), (4) and 
in the examples in Remark~1 (3) 
can be replaced by $o((\log(1/\rho)^{-1})$.

\smallskip

(7)\quad 
It is desirable to show that 
the conditions (A-C) of $F$ imply
the existence of $F_0$ in the assumption of the theorem
(in other words, 
the existence of $F_0$ might be removed 
in the assumption).
Even if the existence of 
$F_0$ is not known, 
we can give the following estimate:
\begin{equation*}
{\mathcal B}_F(\rho)\leq 
\frac{C(F_*)}{\rho^{2/d_F+2}(\log \rho)^{m_F-1}}
\cdot
(1+C\rho^{\varepsilon})
\quad \mbox{for  $\rho\in(0,\delta)$,} 
\end{equation*}
where  
$C(F_*)$ is as in the theorem
and $C, \varepsilon, \delta$ are positive constants. 
This can be easily seen from the proof of Theorem~2.
\end{remark}
%%%%%%%%%%%%%%%%%%%%%%%

%%%%%%%%%%%%%%%%%%%%%%%%%%%%%%
\section{Integral formula of the Bergman kernel}

Let $U$ be a domain in $\C^n$ and 
let
$F:U\to\R_+$ be a $C^{\infty}$-smooth plurisubharmonic 
function.
The weighted Hilbert space $H_{\tau}(U)$ 
($\tau>0$) consists of all entire functions
$\psi:U\to \C$ such that
$$
\int_U |\psi(z)|^2 e^{-2\tau F(z)} dV(z)<\infty,
$$
where $dV(z)$ denotes the Lebesgue measure
on $\C^n$.
We only consider the case where 
$H_{\tau}(U)$ is a nontrivial Hilbert space with
the reproducing kernel. 
(When $U$ is bounded, this is obvious. 
When $U$ is unbounded, the condition (D) 
in Section~2.2 
can imply the above nontriviality.)
We denote the above kernel (on the diagonal) 
by $K_F(z;\tau)$. 
We remark that the function
$\tau\mapsto K_F(z;\tau)$ is continuous
for fixed $z\in U$ from the result in \cite{DiO91}.
F. Haslinger \cite{Has94}, \cite{Has98} 
shows that the Bergman kernel $B_{\Omega_F}(z_0,z)$
of the domain $\Omega_F$ 
can be expressed by the use of $K_F(z;\tau)$ as follows. 
\begin{equation}\label{eqn:5}
B_{\Omega_F}(z_0,z)
=
\frac{1}{2\pi}
\int_0^{\infty} e^{-2\rho \tau}
K_F(z;\tau)\tau d\tau,
\end{equation}
where $\rho$ is the imaginary part of $z_0$.

In the case where $F$ safisfies the conditions (C), (D)
in Section~2.2,  
we can take a complete orthonormal system
for $H_{\tau}(U)$ as 
\begin{equation*}
\left\{
\frac{z^{\alpha}}{c_{\alpha}(\tau)}:
\alpha\in\Z_+^n
\right\},
\quad
\mbox{ with }
c_{\alpha}(\tau)^2=\int_U |z|^{2\alpha} 
e^{-2\tau F(z)}dV(z).
\end{equation*}
Therefore, $K_F(z;\tau)$ can be expressed as
\begin{equation}\label{eqn:6}
K_F(z;\tau)=\sum_{\alpha\in\Z_+^n} 
\frac{|z|^{2\alpha}}{c_{\alpha}(\tau)^2}.
\end{equation}
From (\ref{eqn:5}), (\ref{eqn:6}), 
${\mathcal B}_F(\rho)$ can be expressed as
\begin{equation}\label{eqn:7}
{\mathcal B}_F(\rho)=
\frac{1}{2\pi}\int_0^{\infty} 
e^{-2\rho\tau}
\frac{\tau}{{c}_0(\tau)^2}
d\tau.
\end{equation}
In order to see the behavior of
${\mathcal B}_F(\rho)$ as $\rho\to 0$, 
it suffices to investigate that of 
${c}_0(\tau)^2$ as $\tau\to\infty$.

%%%%%%%%%%%%%%%%%%%%%%%%%%%%
%%%%%%%%%%%%%%%%%%%%%%%%%%%%%%%%

\section{Localization lemma}

Let $U$ be an open neighborhood of 
the origin in $\C^n$ and let 
$F:U\to\R$ be a nonnegative $C^{\infty}$ function 
with $F(0)=0$. 
%Let ${\mathcal B}_F(\rho)$ be as in (\ref{eqn:}).

%%%%
Let $R$ be a positive number such that 
$D(R)\subset U$.
Let $\varphi:\R^n\to[0,1]$ be a $C^{\infty}$ function 
such that 
$\varphi$ identically equals $1$ on 
$D_{\R}(R/2)$ and 
its support is contained in $D_{\R}(R)$. 
Let $\Phi:\C^n\to [0,1]$ be a $C^{\infty}$ function
defined by
$\Phi(z_1,\ldots,z_n)=\varphi(|z_1|,\ldots,|z_n|)$.

We define an integral of the form
\begin{equation}\label{eqn:8}
\tilde{\mathcal B}_F(\rho)=
\frac{1}{2\pi}
\int_0^{\infty} 
e^{-2\rho\tau}
\frac{\tau}{\tilde{c}_0(\tau)^2}
d\tau \quad \mbox{ for $\rho>0$},
\end{equation}
where
\begin{equation}\label{eqn:9}
\tilde{c}_0(\tau)^2=
\int_{U} e^{-2\tau F(z)} 
\Phi(z)dV(z).
\end{equation}

Since
there exists a positive number
$a$ such that $F(z) \leq a\|z\|$ on 
$D(R)$, we see 
\begin{equation}\label{eqn:10}
\begin{split}
\tilde{c}_0(\tau)^2&\geq
\int_{D(R/2)}e^{-2\tau F(z)}dV(z) \\
&\geq 
c\int_0^{R/2} e^{-2a\tau x} x^{2n-1}dx 
\geq 
\frac{c}{(2a\tau)^{2n}}\int_0^1 e^{-s}s^{2n-1}ds
=\frac{C}{\tau^{2n}},
\end{split}
\end{equation}
for $\tau\geq 1$, 
where $c, C$ are positive constants independent 
of $\tau$. 

By the use of the estimate (\ref{eqn:10}), 
the integral $\tilde{\mathcal B}_F(\rho)$ in (\ref{eqn:8})
can be considered as 
a $C^{\infty}$ function defined on $(0,\infty)$, 
which is samely denoted 
by $\tilde{\mathcal B}_F(\rho)$.

Since 
$\tilde{c}_0(\tau)^2 \leq c_0(\tau)^2$, 
the relationship between 
$\tilde{\mathcal B}_F(\rho)$ and 
the Bergman kernel ${\mathcal B}_F(\rho)$ 
can be seen: 
$\tilde{\mathcal B}_F(\rho)\geq{\mathcal B}_F(\rho)$.
%%%%%%%%%%%%%%%%%%%%%%%%
The following proposition shows that 
the singularities of ${\mathcal B}_F(\rho)$ and
$\tilde{\mathcal B}_F(\rho)$ at the origin  
are essentially the same. 

\begin{proposition}
If $F$ satisfies the conditions (A--D) in 
Section~2.2, then 
$\tilde{\mathcal B}_F(\rho)-{\mathcal B}_F(\rho)$ 
can be real-analytically extended to
an open neighborhood of $\rho=0$. 
\end{proposition}

\begin{proof}
From the integral expressions in 
(\ref{eqn:7}) and (\ref{eqn:8}), 
we have
\begin{equation*}
{\mathcal B}_F(\rho)-\tilde{\mathcal B}_F(\rho)
=\frac{1}{2\pi}\int_0^{\infty} 
e^{-2\rho\tau}
\left(
\frac{1}{c_0(\tau)^2}-
\frac{1}{\tilde{c}_0(\tau)^2}
\right)\tau d\tau.
\end{equation*}
Therefore, 
the proposition can be easily seen by the use of 
the following lemma. 
\end{proof}

%%%%%%%%%%%%%%%%%
\begin{lemma}
There exist positive numbers $L, C, q, \epsilon$ 
such that 
$$
\left|
\frac{1}{\tilde{c}_0(\tau)^2}-
\frac{1}{{c}_0(\tau)^2}
\right|
\leq C \tau^{q} e^{-2\epsilon \tau}
\quad \mbox{ for $\tau\geq L$.}
$$
\end{lemma}

\begin{proof}
Let 
\begin{equation*}
e(\tau):={c}_0(\tau)^2-\tilde{c}_0(\tau)^2
=\int_{U} e^{-2\tau F(z)}(1-\Phi(z)) dV(z).
\end{equation*}

From the conditions (A), (C) in Section~2.2,
there exist positive numbers $c, \epsilon$ such that
\begin{equation*}%\label{eqn:}
F(z)\geq 
\epsilon+c(|z_1|^{\beta}+\cdots+|z_n|^{\beta})
\quad \mbox{for $z\in U\setminus D(R/2)$}. 
\end{equation*}
By the use of the above inequality, we can 
\begin{equation}\label{eqn:11}
\begin{split}
|e(\tau)|&\leq
\int_{U \setminus D(R/2)}
e^{-2\tau F(z)} dV(z) \\
&\leq 
(2\pi)^ne^{-2\epsilon \tau} 
\int_{\R_+^n}
e^{-2c\tau(x_1^{\beta}+\cdots+x_n^{\beta})}
\left(
\prod_{j=1}^n x_j 
\right)
dx 
\leq 
C \tau^{-2n/\beta} e^{-2\epsilon\tau}.
\end{split}
\end{equation}

%%%%%%%%%%%%%%%%%%%%%%%%%%%%%%%%%%%%%
Applying (\ref{eqn:10}), (\ref{eqn:11}) 
to the right hand side of the equation
\begin{equation*}
\frac{1}{\tilde{c}_0(\tau)^2}-
\frac{1}{{c}_0(\tau)^2}
=
\frac{e(\tau)}{\tilde{c}_0(\tau)^4
(1+e(\tau)/\tilde{c}_0(\tau)^2)},
\end{equation*}
we can get the estimate in the lemma. 
\end{proof}

%%%%%%%%%%%%%%%%%%%%%%%%%%%%

\section{The $\hat{\mathcal E}$-condition}
 
%%%%%
In this section, we introduce some classes of 
$C^{\infty}$ functions, 
which are defined by the use of 
Newton polyhedra. 
When $F$ belongs to this class, 
the behavior of $\tilde{\mathcal B}_F(\rho)$
is relatively easy to be understood 
(see Theorem~3, below).
 
%%%%%%%%%%%%%%%%%%%%%%%%%%%%%%%%%%%%%%
\subsection{Newton polyhedra in the real case}
Let $f$ be a real-valued $C^{\infty}$ function
defined near the origin in $\R^n$. 
Let 
\begin{equation}\label{eqn:12}
\sum_{\alpha\in\Z_+^n} 
c_{\alpha}x^{\alpha}
=\sum_{\alpha\in\Z_+^n}
c_{\alpha}x_1^{\alpha_1}\cdots x_n^{\alpha_n}
\end{equation}
be 
the Taylor series of $f$ at the origin. 
The Newton polyhedron
of $f$ is the integral polyhedron
\begin{equation*}
{\mathcal N}_+(f)=
\mbox{the convex hull of the set
$\displaystyle 
\bigcup\{\alpha+\R_+^n:
\alpha\in S(f)$\} in $\R_+^n$},
\end{equation*}
where $S(f):=\{\alpha\in\Z_+^n: c_{\alpha}\neq 0\}$.
In the real case, we also say 
that $f$ is {\it flat} if   
${\mathcal N}_+(f)=\emptyset$ and
that $f$ is {\it convenient}
if ${\mathcal N}_+(f)$ intersects 
all the axes.

%%%%%%%%%%%%%%%%%%%%%%%%%%%%

\subsection{The $\hat{\mathcal E}$-condition
in the real case}

Let $f$ be a $C^{\infty}$ function defined near the origin
in $\R^n$.
We say that 
$f$ {\it admits the $\gamma$-part}
on an open neighborhood $V$ 
of the origin in $\R^n$ 
if for any $x\in V$ the limit
\begin{equation*}
\lim_{t\to 0}
\frac{f(t^{a_1} x_1,\ldots, t^{a_n} x_n)}{t^l}
\end{equation*}
exists for {\it all} pairs 
$(a,l)=(a_1,\ldots,a_n,l)\in\Z_+^n\times\Z_+$
defining $\gamma$ 
(i.e., 
$\{x\in P:\sum_{j=1}^n a_j x_j =l\}=\gamma$).
It is known in \cite{KaN16jmst} that
when $f$ admits the $\gamma$-part, 
the above limits take the same value
for any $(a,l)$, which is denoted by 
$f_{\gamma}(x)$. 
We consider $f_{\gamma}$ as the function
on $V$, which is called the $\gamma$-{\it part}
of $f$.
For a compact face $\gamma$ of 
${\mathcal N}_+(f)$, $f$ always admits
the $\gamma$-part near the origin 
and $f_{\gamma}(x)$ equals the polynomial
$\sum_{\alpha\in\gamma\cap\Z_+^n}
c_{\alpha} x^{\alpha}$, where $c_{\alpha}$ 
are as in (\ref{eqn:12}).

%%%%%%%%%%%%%%%%%%%%%%%%

\begin{definition}
[\cite{KaN16jmst}, \cite{KaN16tams}]
We say that $f$ satisfies the 
$\hat{\mathcal E}$-{\it condition 
at the origin} 
if $f$ admits the $\gamma$-part 
for every proper face $\gamma$ of 
${\mathcal N}_+(f)$.
\end{definition}

\begin{remark}
(1)\quad 
If $f$ is real analytic near the origin 
or $f$ is convenient, 
then $f$ satisfies the $\hat{\mathcal E}$-condition
(see \cite{KaN16jmst}, \cite{KaN16tams}).
In particular, 
in the one-dimensional case, 
nonflat functions satisfy the $\hat{\mathcal E}$-condition.

\smallskip

(2)\quad
For example, 
$f(x_1,x_2)=x_1^2+e^{-1/x_2^2}$ does not 
satisfy the $\hat{\mathcal E}$-condition. 
\end{remark}

%%%%%%%%%%%%%%%%%%%%%%%

%%%%%%%%%%%%%%%%%%%%%%%

Every $C^{\infty}$ function $f$
can be decomposed into 
an $\hat{\mathcal E}$-function 
and a flat function.

%%%%%%%%%%%%%%%%%%%%%%%
\begin{proposition}
For any $C^{\infty}$ function $f$ 
defined near the origin in $\R^n$, 
there exists a $C^{\infty}$ function $f_0$ 
satisfying the $\hat{\mathcal E}$-condition 
at the origin such that $f-f_0$ is flat 
at the origin. 
%Furthermore, 
%if $f$ is positive away from the origin, then
%$f_0$ and $e$ can be taken as nonnegative functions. 
\end{proposition}
%%%%%%%%%%%%%%%%%%%%%%%%%

\begin{proof}
Since the proposition is obvious when
$f$ does not vanish at the origin or 
$f$ is flat at the origin, 
we will only consider the other cases. 

Let $p=(p_1,\ldots,p_n)\in\Z_+^n$ be a 
vertex of the Newton polyhedron of $f$. 
We inductively 
define $C^{\infty}$ functions $R_0,R_1,\ldots,R_n$ 
defined near the origin as follows.

Let $R_0(x)=f(x)$. Let $k$ is an integer with 
$1\leq k\leq n$.
If $p_k\geq 1$,  then 
there exists a $C^{\infty}$ function $R_{k}$ 
defined near the origin such that  
\begin{equation}\label{eqn:13}
R_{k-1}(x)=
\sum_{j=0}^{p_k-1}
c_{kj}(x_1,\ldots,x_{k-1},x_{k+1},\ldots,x_n)x_k^j+
x_k^{p_k} R_k(x),
\end{equation}
where $c_{jk}$ are $C^{\infty}$ functions 
of $(n-1)$-variables. 
Note that (\ref{eqn:13}) is the Taylor expansion 
of $R_{k-1}$
with respect to the variable $x_k$.
On the other hand, 
if $p_k=0$, then set $R_k(x)=R_{k-1}(x)$.

By the use of the equations for $k=1,\ldots,n$, 
$f$ can be expressed as 
\begin{equation}\label{eqn:14}
f(x)=P(x)+x^p R_n(x),
\end{equation}
where $P$ is written in the form
of the sum of a monomial times 
a $C^{\infty}$ function of $(n-1)$-variables, 
and $x^p=x_1^{p_1} \cdots x_n^{p_n}$. 
It is easy to see that $R_n(0)=c_p$, 
where $c_p$ is as in (\ref{eqn:12})
(i.e., $c_p$ is the coefficient of the term 
containing $x^p$ of the Taylor series of $f$). 
Since $p$ is the vertex of the Newton polyhedron
of $f$, $c_p$ does not vanish, which implies
$R_n(0)\neq 0$.
It follows from Proposition~6.3 
in \cite{KaN16jmst}, we can see that $x^p R_n(x)$ 
satisfies the 
$\hat{\mathcal E}$-condition. 

Now, let us construct a $C^{\infty}$ function
$f_0$ in the proposition by induction on $n$.
Note that every non-flat $C^{\infty}$ function of one variable 
satisfies the $\hat{\mathcal E}$-condition. 
Let $f$ be a non-flat $C^{\infty}$ function of $n$-variables
with $f(0)=0$. 
Then $f$ can be expressed as in (14) and 
$P$ takes the form of the sum of 
a monomial times a $C^{\infty}$ function
of $(n-1)$-variables.
Now we assume that every $C^{\infty}$ function
of $(n-1)$-variables can be decomposed 
into an $\hat{\mathcal E}$-function 
and a flat function. 
Then 
$P$ has the same type decomposition
and so a desirable $\hat{\mathcal E}$-function 
$f_0$ can be obtained.

\end{proof}

%%%%%%%%%%%%%%%%%%%%%%%%%%%%%%%%%%%%%%%%%%%%%
\subsection{The $\hat{\mathcal E}$-condition
in the complex case}

%%%%%%%%%%%%%%%%%%%%%%%%
Now, let us consider a nonnegative 
$C^{\infty}$ function $F$ defined 
near the origin in $\C^n$ with $F(0)=0$.  
If $F$ satisfies the condition (C), then
there exists a $C^{\infty}$ function $f$
defined near the origin in $\R^n$ such that 
$f(|z_1|,\ldots,|z_n|)=F(z_1,\ldots,z_n)$.
We say that $F$ satisfies the $\hat{\mathcal E}$-condition
at the origin,  
if $f$ satisfies the $\hat{\mathcal E}$-condition 
at the origin. 
It is easy to see that 
the $\hat{\mathcal E}$-condition of $F$ 
is independent of the choice of the function $f$.  

\begin{proposition}
For any $C^{\infty}$ function $F$ 
defined near the origin in $\C^n$
satisfying the condition (C), 
there exists a $C^{\infty}$ function $F_0$ 
satisfying the condition (C) and 
the $\hat{\mathcal E}$-condition 
at the origin such that $F-F_0$ is flat 
at the origin. 
%Furthermore, 
%if $f$ is positive away from the origin, then
%$f_0$ and $e$ can be taken as nonnegative functions. 
\end{proposition}

\begin{remark}
It is desirable to show that
the additional conditions (A), (B) imply
the existence of $F_0$ such that  
$F-F_0$ is a {\it nonnegative} flat function. 
If this implication can be shown, 
then the existence of 
$F_0$ can be removed in the assumption
in Theorem~2.
\end{remark}

%%%%%%%%%%%%%%%%%%%%%%
%%%%%%%%%%%%%%%%%%%%%%
\section{Behavior of $\tilde{\mathcal B}_F(\rho)$ 
in the $\hat{\mathcal E}$-case}
%%%%%%%%%%%%%%%%%%%%%%%%%%%%%%%%%
In order to prove Theorem~2, 
it suffices to consider the case where
$F$ satisfies the $\hat{\mathcal E}$-condition. 
%%%%%%%%%%%%%%%%%%%%%%%%%%%%
\begin{theorem}
Suppose that $F$ is a $C^{\infty}$ function 
with $F(0)=0$ satisfying the condition (C) 
in Section~2.2 
and the $\hat{\mathcal E}$-condition
at the origin.   
If the principal face of ${\mathcal N}_+(F)$ 
is compact, then 
\begin{equation*}\label{eqn:}
\tilde{\mathcal B}_F(\rho)=
\frac{C(F_*)}{\rho^{2+2/d_F}(\log\rho)^{m_F-1}}
\cdot
(1+o(\rho^{\varepsilon})) \quad 
\mbox{ as $\rho\to 0$,}
\end{equation*}
where $d_F$ and $m_F$ are as in Section~2.1, 
$\varepsilon$ is a positive constant, and 
$C(F_*)$ is a positive constant depending only 
on the principal part $F_*$ of $F$. 
\end{theorem}
%%%%%%%%%%%%%%%%%%%%%%%%%%%%%

\begin{proof}
Recall the expression of $\tilde{\mathcal B}_F(\rho)$:
\begin{equation*}\label{eqn:}
\begin{split}
&\tilde{\mathcal B}_F(\rho)=
\frac{1}{2\pi}
\int_0^{\infty} 
e^{-2 \rho \tau} \frac{\tau}{\tilde{c}_0(\tau)^2} d\tau\\
&\mbox{ with }\,\,
\tilde{c}_0(\tau)^2=
\int_{\C^n} e^{-2\tau F(z)}\Phi(z)dV(z).
\end{split}
\end{equation*}
This theorem can be directly shown 
by applying the following lemma to
the above integral formulas. 
\end{proof}
%%%%%%%%%%%%%%%%
\begin{lemma}
%%%%
Under the same assumptions on $F$ as those 
of Theorem~3, 
$\tilde{c}_0(\tau)^2$ satisfies 
\begin{equation}\label{eqn:15}
\tilde{c}_0(\tau)^2 =
c(F_*)\tau^{-2/d_F} (\log \tau)^{m_F-1} \cdot
(1+o(\tau^{-\varepsilon}))\quad 
\mbox{ as $\tau\to\infty$},
\end{equation}
where $\varepsilon$ is a positive constant and 
$c(F_*)$ is a positive constant depending only on 
the principal part $F_*$ of $F$.
\end{lemma}
%%%%%%%%%%%%%%%%%%%%%%
The proof of the above lemma will be given in Section~8.

%%%%%%%%%%%%%%%%%%%%%%%%%%%%%%%%%%%%%%%

\section{Proof of Theorem~2}

%%%%%%%%%%%%%%%%%%%%%%%%%%%
Let us give a proof of Theorem~2. 
We assume that $F$ is a $C^{\infty}$ function satisfying 
the conditions (A--D) in Section~2.2. 
Moreover,  let $F_0$ be as in Theorem~2.
It is easy to see that
$F_0$ also satisfies 
the conditions (B), (C). 
%%%%

For a positive number $M$, we define 
%%% 
\begin{equation*}
F_1(z)=F_0(z)+ \sum_{j=1}^n |z_j|^{2M}.
\end{equation*} 
Since the principal face of ${\mathcal N}_+(F_0)$ 
is compact, there exists a positive constant $M$ such that
the principal face of ${\mathcal N}_+(F_1)$
is the same as that of ${\mathcal N}_+(F_0)$.  
Moreover, since $F_1$ is convenient, 
$F_1$ also satisfies the 
$\hat{\mathcal E}$-condition 
at the origin (see Remark~2~(1) ).
%%%
It is easy to see that 
there exists a positive number $R$ such that 
\begin{equation}\label{eqn:16}
F_0(z)\leq F(z)\leq F_1(z) \quad 
\mbox{ for $z\in D(R)$.}
\end{equation}

For $j=0,1$, we define
\begin{equation*}
\tilde{\mathcal B}_{F_j}(\rho)=
\frac{1}{2\pi}
\int_0^{\infty} 
e^{-2 \rho \tau} \frac{\tau}{g_j(\tau)} d\tau
 \quad \mbox{ for $\rho>0$,}
\end{equation*}
where
\begin{equation*}
g_j(\tau)=
\int_{\C^n} e^{-2\tau F_j(z)}\Phi(z)dV(z),
\end{equation*}
where $\Phi$ is as in Section~4. 
From the relationship in (\ref{eqn:16}), 
we see
\begin{equation}\label{eqn:17}
\tilde{\mathcal B}_{F_0}(\rho)\leq 
\tilde{\mathcal B}_F(\rho)\leq 
\tilde{\mathcal B}_{F_1}(\rho).
\end{equation}
Since it is easy to see that 
$F_0,F_1$ satisfy the assumptions in 
Theorem~3,  
we have 
\begin{equation}\label{eqn:18}
\tilde{\mathcal B}_{F_j}(\rho)=
\frac{C(F_*)}{\rho^{2+2/d_F}(\log\rho)^{m_F-1}}\cdot
(1+o(\rho^{\varepsilon_j})) \quad 
\mbox{ as $\rho\to 0$},
\end{equation}  
for $j=0,1$, where $\varepsilon_0, \varepsilon_1$ 
are positive numbers.
Therefore, (\ref{eqn:17}), (\ref{eqn:18}) imply that 
\begin{equation*}
\tilde{\mathcal B}_F(\rho)=
\frac{C(F_*)}{\rho^{2+2/d_F}(\log\rho)^{m_F-1}}\cdot
(1+o(\rho^{\varepsilon})) \quad 
\mbox{ as $\rho\to 0$,}
\end{equation*}
where $\varepsilon$ is the minimum of 
$\varepsilon_1, \varepsilon_2$.
From the localization lemma in Proposition~1,
%the localization lemma (Proposition , below):
%${\mathcal B}_F(\rho)-\tilde{{\mathcal B}_F}(\rho)$
%is real analytic near the origin, 
we can obtain the equation 
(\ref{eqn:4}) in the theorem.

%%%%%%%%%%%%%%%%%%%%%%

\section{
Asymptotic analysis of
some Laplace integrals}

%%%%%%%%%%%%%%%%%%%%%%%%%%%%%%%%%%%%%%%
%%%%%%%%%%%%%%%%%%%%%%%%%%%%%%%

The purpose of this section is to give a
proof of Lemma~3. 
For this purpose, we will investigate 
the behavior of a Lapace integral of the form 
\begin{equation}\label{eqn:19}
L_f(\tau)=
\int_{\R_+^n} 
e^{-2\tau f(x)}\varphi(x) 
\left(\prod_{j=1}^n x_j \right) dx, 
\end{equation}
for large $\tau>0$, 
where 
$f$, $\varphi: V \to\R_+$ are 
nonnegative and nonflat 
$C^{\infty}$ functions defined on 
a small open neighborhood $V$ of the origin
with $f(0)=0$ and $\varphi(0)=1$
and the support of $\varphi$ is contained in 
$V$ of the origin.
%%%
Since the compact support of $\varphi$
implies the convergence of the
above integral, 
$L_f$ can be considered as 
a $C^{\infty}$ function defined on $(0,\infty)$.

The analysis in this section is based on 
the studies in \cite{CKN11}, \cite{KaN16jmst}, 
\cite{KaN16tams}, 
which deal with oscillatory integrals 
and local zeta functions.

%%%%%%%%%%%%%%%%%%%%%%%%%%%%%%%%%%%%%%%
\subsection{Newton data in the real case}

Let $f$ be a nonflat real-valued 
$C^{\infty}$ function defined near the origin in $\R^n$.
Let $f$ admit the Taylor series 
$\sum_{\alpha\in\Z_+^n} c_{\alpha} x^{\alpha}$
at the origin and  
let ${\mathcal N}_+(f)$ be the Newton polyhedron
of $f$ defined in Section~5.1. 
The {\it Newton distance} of $f$ is a nonnegative
number 
\begin{equation*}
d_f=\min\{
t\geq 0:
(t,\ldots,t)\in{\mathcal N}_+(f)
\}.
\end{equation*}
The minimum proper face of 
${\mathcal N}_+(f)$ containing the point
$P_f:=(d_f,\ldots,d_f)$
is called the 
{\it principal face} of ${\mathcal N}_+(f)$,
and is denoted by $\gamma_*$.
The codimension of $\gamma_*$ 
is called the {\it Newton multiplicity}, 
which is denoted by $m_f$.
When $\gamma_*$ is compact, 
the {\it principal part} of $f$ is defined by 
$
f_*(x)=
\sum_{\alpha\in\gamma_*\cap\Z_+^n}
c_{\alpha} x^{\alpha}.
$
%where $c_{\alpha}$ are as in (\ref{eqn:12}).

%%%%%%%%%%%%%%%%%%%%%%%%%%%%%%
\subsection{Meromorphic extension 
of some local zeta functions}

Let us consider an integral of the form
\begin{equation}\label{eqn:20}
Z_f(s)=\int_{\R_+^n}
f(x)^s \varphi(x) \left(\prod_{j=1}^n x_j \right) dx
\quad \mbox{ for $s\in\C$},
\end{equation}
where $f, \varphi$ are the same as in (\ref{eqn:19}).
The convergence of the integral easily implies that
$Z_f$ can be regarded as a holomorphic function 
on the right half-plane, which is called a
{\it local zeta function} and is samely denoted 
by $Z_f$.

The situation of analytic extension of 
the above local zeta function plays crucial
roles in the investigation of the behavior of 
Laplace integrals as $\tau\to\infty$.

%%%%%%%%%%%%%%%%%%%%%%%%%%%%%%%%%%%%
\begin{theorem}
Suppose that 
\begin{enumerate}
\item $f$ satisfies the $\hat{\mathcal E}$-condition (see Section~5);
\item $f$ is Newton nondegenerate (see Section 9.3);
\item The principal face of ${\mathcal N}_+(f)$ is 
compact.
\end{enumerate}
Then $Z_f$ can be analytically continued as a
meromorphic function to the whole complex plane, 
which will be samely denoted by $Z_f$. 
More precisely, 
the set of the poles of $Z_f$ are contained 
in the set $\{-j/m: j\in\N\}$ 
where $m$ is a positive integer.  

Furthermore, 
the most right pole of $Z_f$
exists at $s=-2/d_f$ and its order is $m_f$, 
where $d_f, m_f$ are as in Section~8.1.
The coefficient of the leading term of 
the Laurent expansion of $Z_f$ at $s=-2/d_f$ 
\begin{equation*}
c_{{\tiny Z}}(f_*)=\lim_{s\to -2/d_f} 
\left(
s+\frac{2}{d_f}
\right)^{m_f} \cdot Z_f(s)
\end{equation*}
is a positive constant depending only on the 
principal part $f_*$ of $f$.
\end{theorem}
%%%%%%%%%%%%%%%%%%%%%%%%%%%%%%%%%%%%%%%

\begin{proof}
The above theorem is a special case of 
Theorems 10.1 and 10.8 
in \cite{KaN16tams}.
Indeed, 
$f$ satisfies the conditions (b) and (c) in Theorem 10.8.

\end{proof}

\begin{remark}
The idea of the proof is based on 
the method of Varchenko 
in \cite{Var76}, \cite{AGV88}
(see also \cite{CKN11}, \cite{KaN16jmst}, 
\cite{KaN16tams}).
In his method, the toric resolution of 
singularities based on the geometry
of Newton polyhedra plays an essential role. 
\end{remark}

%%%%%%%%%%%%%%%%%%%%%%%%%%%%%%%%%%%% 
\subsection{Asymptotic behavior of some Laplace integrals}

From an information about the analytic continuation of 
$Z_f$ in Theorem~4, we can see the behavior of 
the Laplace integral in (\ref{eqn:19}) 
as $\tau\to\infty$.

%Let us consider a Laplace integral of the form in (\ref{eqn:})

%%%%%%%%%%%%%%%%%%%%%%%%%%%%%%%%%%%%%

\begin{theorem}
Suppose that $f,\varphi$ satisfy the same conditions 
as in Theorem~4.
%%%%%%%%%%%%%%
Then $L_f$ admits the asymptotic expansion: 
for any $N\in\N$, 
there exists a positive constant $C_N$ 
such that 
\begin{equation}\label{eqn:21}
\left|
L_f(\tau)-\sum_{j=0}^{N}\sum_{k=1}^{n} 
c_{jk} \tau^{-j/m} 
(\log\tau)^{k-1}
\right|<C_N \tau^{-N/m-\varepsilon},
\end{equation}
for $\tau\geq 2$, 
where $m$ is a positive integer determined by 
${\mathcal N}_+(f)$, 
$\varepsilon$ is a positive number and  
$c_{jk}$ are constants.
%%%%%%%%%%%%%%
In particular, the first term of the expansion 
can be expressed as 
\begin{equation*}
L_f(\tau)=
c_{L}(f_*)\tau^{-2/d_f}(\log\tau)^{m_f-1}
\cdot (1+o(\tau^{-\varepsilon}))\quad
\mbox{ as $\tau\to\infty$},
\end{equation*}
where $d_f$ and $m_f$ are as in Section~8.1
and  
$c_{L}(f_*)$ is a positive number depending only on 
the principal part $f_*$ of $f$. 
\end{theorem}

%%%%%%%%%%%%%%%%%%%%%%%%%%%%%%%%%%%%%%

\begin{proof}
Define
the fiber integral $H_f:\R_+ \to\R$ as
\begin{equation*}
H_f(u)=\int_{W_u}\varphi(x)\left(\prod_{j=1}^n x_j \right)
\omega,
\end{equation*}
where $W_u:=\{x\in\R^n: f(x)=u\}$ and 
$\omega$ is the surface element on $W_u$, 
which is determined by 
$df\wedge \omega=
d x_1\wedge \cdots \wedge d x_n$. 

It is easy to see that 
the Laplace integral $L_f$ 
and the local zeta function $Z_f$ can be represented
by the use of $H_f$ as follows. 
\begin{equation}\label{eqn:22}
L_f(\tau)=\int_0^{\infty} e^{-2\tau u} H_f(u)du,
\end{equation}
\begin{equation}\label{eqn:23}
Z_f(s)=\int_0^{\infty} u^s H_f(u)du.
\end{equation}
Applying the inverse formula 
of the Mellin transform to (\ref{eqn:23}), 
we have
\begin{equation*}
H_f(u)=\frac{1}{2\pi i} 
\int_{r-i\infty}^{r+i\infty} 
Z_f(s) u^{-s-1}ds,
\end{equation*}
where $r>0$ and the integral contour follows
the line ${\rm Re}(s)=r$ upwards. 
The meromorphic extension 
of $Z_f$ in Theorem~4 implies that the deformation 
of the integral contour 
as $r$ tends to $-\infty$ gives 
an asymptotic expansion of $H_f(u)$ as $u\to 0$ 
by the residue formula. 
For any $N\in\N$, there exists a positive constant 
$B_N$ such that 
\begin{equation}\label{eqn:24}
\left|
H_f(u)-\sum_{j=0}^{N}\sum_{k=1}^{n} 
b_{jk} u^{j/m} 
(\log u)^{k-1}
\right|< B_N u^{N/m+\varepsilon},
\end{equation}
for $u\in(0,1/2)$, 
where $m$ is a positive integer determined by 
${\mathcal N}_+(f)$, 
$\varepsilon$ is a positive number and 
$b_{jk}$ are constants.
We remark that the above deformation 
of the contour can be done. 
By applying the asymptotic expansion (\ref{eqn:24})
to (\ref{eqn:22}), 
we can obtain (\ref{eqn:21}). 
%%%%%%%%%%%%%%%%%%%

In particular, 
we have 
\begin{equation}\label{eqn:25}
H_f(u)=c_{H}(f_*) 
u^{2/d_f-1} (\log u)^{m_f-1}\cdot(1+o(u^{\varepsilon})) 
\quad \mbox{ as $u\to 0$.}
\end{equation}
Substituting (\ref{eqn:25}) into (\ref{eqn:22}), we have
\begin{equation*}
L_f(\tau)=c_{L}(f_*) 
\tau^{-2/d_f}(\log \tau)^{m_f-1}
\cdot(1+o(\tau^{-\varepsilon}))
\quad \mbox{ as $\tau\to\infty$.}
\end{equation*} 
\end{proof}

%%%%%%%%%%%%%%%%%%%%%%%%%%%%%%
\subsection{Proof of Lemma 3}

By the use of the polar coordinates 
$z_j=x_j e^{i\theta_j}$, with 
$x_j\geq 0, \theta_j\in\R$, for 
$j=1,\ldots,n$, then we have 
\begin{equation*}
\begin{split}
\tilde{c}_0(\tau)^2&=
\int_{U} e^{-2\tau F(z)} 
\Phi(z)dV(z) \\
&=
(2\pi)^n \int_{\R_+^n} e^{-2 \tau f(x)} 
\left(\prod_{j=1}^n x_j \right) \varphi(x)dV(x),
\end{split}
\end{equation*}
where 
$\varphi$ is as in Section~4 and $f$ 
is a $C^{\infty}$ function defined near the origin 
satisfying 
$f(|z_1|, \ldots, |z_n|)=F(z_1,\ldots, z_n)$. 

From \cite{Kam21jmsj}, the conditions (A-C) implies 
that $F_{\gamma}$ is positive on 
$(\R\setminus\{0\})^n$ for every 
compact face of ${\mathcal N}_+(F)$, 
which implies the Newton nondegeneracy 
of $f$ (see Section~9.3). 
Moreover, it follows from the assumptions of $F$ that 
the the principal face of $f$ is compact  
and $f$ satisfies the $\hat{\mathcal E}$-condition. 
Therefore, since 
$f$ satisfies all the assumptions in Theorem~5, 
we can obtain 
\begin{equation*}
\tilde{c}_0(\tau)^2
=
C(f_*) \tau^{-2/d_f} (\log \tau)^{m_f-1}\cdot
(1+o(\tau^{-\varepsilon})),
\end{equation*}
where $C(f_*)$ is a positive constant depending only
on $f_*$ and $\varepsilon$ is a positive number. 
Since $d_f=d_F$ and $m_f=m_F$, 
we can obtain (\ref{eqn:15}) in the lemma.

%%%%%%%%%%%%%%%%%%%%%%%%%%%%%%

\section{Appendix}
%\addcontentsline{toc}{section}{Appendix}

%\appendix
%\chapter{Convex geometry} 
%\section{Convex geometry}

%%%%%%%%%%%%%%%%%%%%%%%%%%%%%%%%%

\subsection{Asymptotic expansion in (\ref{eqn:2})}

Theorem~1 implies that the singularity 
of the Bergman kernel can be expressed 
by the following complicated asymptotic 
expansion
%%%% 
\begin{equation*}
\Psi(\rho) \sim \sum_{j=0}^{\infty}
\sum_{k=a_j}^{\infty} 
C_{jk}\frac{\rho^{j/m}}{(\log(1/\rho))^{k}}
\quad \mbox{as $\rho\to 0$,} 
\end{equation*}
where $m$ is a positive integer, 
$a_j$ are integers and 
$C_{jk}$ are real numbers.
%%%

Let us explain the exact meaning of 
the above asymptotic expansion. 
%Let $\delta$ be a small positive number. 
For any $N\in\N$, 
there exists a positive number
$C_N$ such that 
$$
\left|
\Psi(\rho) - \sum_{j=0}^{N} 
C_{j}(\zeta) \rho^{j/m} 
\right|
<C_N \rho^{N/m+\varepsilon} 
\quad \mbox{ for $\rho\in(0,\delta)$},
$$
where $\zeta=\log(1/\rho)^{-1}$ and 
every $C_j(\zeta)$ satisfies that, for any $M\in\N$
satisfying $M\geq a_j$ there exists a positive number
$\tilde{C}_M$ such that 
$$
\left|
C_j(\zeta) - \sum_{k=a_j}^{M} 
C_{jk} \zeta^{k} 
\right|
<\tilde{C}_M \zeta^{M+\tilde{\varepsilon}} 
\quad \mbox{ for $\zeta\in(0,\delta)$}.
$$
Here $\varepsilon$, $\tilde{\varepsilon}$, 
$\delta$ are some positive numbers. 
Note that $\zeta\to 0$ if and only if $\rho\to0$.

%%%%%%%%%%%%%%%%%%%%%%%%%%%%%%%%%%
\subsection{Convex geometry}

Let us explain fundamental notions
in the theory of convex polyhedra
which are necessary for our investigation. 
Refer to \cite{Zie95} for a general theory
of convex polyhedra. 

For $(a,l)\in\R^n\times\R$, 
let $H(a,l)$ and $H_+(a,l)$
be a hyperplane and a closed half-space 
in $\R^n$ defied by 
\begin{equation*}
\begin{split}
&H(a,l)=\{x\in\R^n; \langle a, x \rangle=l\}, \\
&H_+(a,l)=
\{x\in\R^n;\langle a, x \rangle\geq l\},
\end{split}
\end{equation*}
respectively. 
Here $\langle a, x \rangle:=\sum_{j=1}^n a_j x_j$.

A ({\it convex rational}) {\it polyhedron} is
an intersection of closed half-spaces:
a set $P\subset \R^n$ presented in the 
form $P=\bigcap_{j=1}^N
H_+(a^j,l_j)$ for some $a^1,\ldots,a^N\in\Z^n$
and $l_1\ldots,l_N\in\Z$. 

Let $P$ be a polyhedron in $\R^n$.
A pair $(a,l)\in\Z\times\Z$ is said to be
valid for $P$ if 
$P$ is contained in $H_+(a,l)$. 
A face of $P$ is any set of the form
$F=P\cap H(a,l)$, where
$(a,l)$ is valid for $P$. 
Since $(0,0)$ is always valid, 
we consider $P$ itself as a trivial
face of $P$; the other faces are called
{\it proper faces}. 
Conversely, it is easy to see that
any face is a polyhedron. 
Considering the valid pair $(0,-1)$, 
we see that the empty set is always
a face of $P$. 
Indeed, $H_+(0,-1)=\R^n$, but 
$H(0,-1)=\emptyset$. 

The dimension of a face $F$
is the dimension of its affine hull 
(i.e., 
the intersection of all affine flats that contain 
$F$), which is denoted by $\dim(F)$.
The faces of dimensions $0,1$ and $\dim(P)-1$
are called vertices, edges and facets, 
respectively. 

When $P$ is the Newton polyhedron of 
some non-flat $C^{\infty}$ function, 
$\gamma$ is a compact face
if and only if every valid pair 
$(a,l)=(a_1,\ldots,a_n,l)$ defining 
$\gamma$ satisfies $a_j>0$ for any $j$.

%%%%%%%%%%%%%%%%%%%%%%%%%
\subsection{Newton nondegeneracy condition}

%Let us explain this condition. 
We say that $f$ is {\it Newton nondegenerate} if 
the gradient of the $\gamma$-part of $f$ 
has no zero in $(\R\setminus\{0\})^n$
for every compact face $\gamma$ 
of the Newton polyhedron 
${\mathcal N}_+(f)$. 
This concept is very important
in the study of singularity theory.

When $\gamma$ is a compact face 
of ${\mathcal N}_+(f)$
with the valid pair $(a_1,\ldots,a_n,l)$ defining 
$\gamma$, 
the following Euler identity holds:
$$
l f_{\gamma}(x)=
a_1 x_1\frac{\partial f_{\gamma}}{\partial x_1}(x)+\cdots+
a_n x_n\frac{\partial f_{\gamma}}{\partial x_n}(x). 
$$
It follows from this identity that  
if $f_{\gamma}$ has no zero
in $(\R\setminus\{0\})^n$
for every compact face $\gamma$, 
then
$f$ is Newton nondegenerate.

\bigskip

{\it Acknowledgement}
%The author is grateful to someone
%for useful discussion.
The author greatly appreciates that the referee 
carefully read 
this paper and gave many valuable comments. 
%which have greatly improved the readability of the paper.
This work was supported by 
JSPS KAKENHI Grant Numbers JP20K03656, JP20H00116

\end{document}